\documentclass[a4paper,10pt]{article}
\usepackage[utf8]{inputenc}
\usepackage[T1]{fontenc}
\usepackage[english]{babel}
\usepackage{amsmath}
\usepackage{amsthm}
\usepackage{amsfonts}
\usepackage{amssymb}
\usepackage{listings}
\usepackage{graphicx}
\usepackage{mathrsfs}
\usepackage{hyperref}
\usepackage{algpseudocode}
\usepackage{algorithmicx}
\usepackage{algorithm}
\usepackage{subfig}
\usepackage{stmaryrd}
\usepackage{bm}
\usepackage{tikz}
\usepackage{braket}
\usepackage{wasysym}
\usepackage{wrapfig}
\usepackage[framemethod=TikZ]{mdframed}
\usepackage{xcolor}
\usepackage{pstricks}
\usepackage{faktor}

\usepackage[top=2.00cm,bottom=3.00cm,left=2.00cm,right=2.00cm]{geometry}

\newcommand{\wt}{\widetilde}
\newcommand{\wh}{\widehat}

\newcommand{\f}{\mathbb}

\newcommand{\R}{\mathbb{R}}

\renewcommand{\P}{\mathbb{P}}
\newcommand{\E}{\mathbb{E}}

\DeclareMathOperator{\rank}{rank}
\DeclareMathOperator{\mvec}{vec}
\DeclareMathOperator{\diag}{diag}

\theoremstyle{definition}
\newtheorem{definition}{Definition}
\theoremstyle{plain}
\newtheorem{theorem}{Theorem}

\newtheorem{lemma}{Lemma}
\newtheorem{remark}{Remark}
\newtheorem{conjecture}{Conjecture}
\newtheorem{corollary}{Corollary}

\title{Role extraction for digraphs via neighbourhood pattern similarity}
\author{Giovanni Barbarino\thanks{Department of Mathematics and Systems Analysis, Aalto University, Finland.},
	Vanni Noferini\thanks{Department of Mathematics and Systems Analysis, Aalto University, Finland.}, 
	Paul Van Dooren\thanks{Department of Mathematical Engineering, Universit\'e catholique de Louvain, Belgium}
}
\date{}

\begin{document}
\maketitle
\begin{abstract}
	We analyse the recovery of different roles in a network modelled by a directed graph, based on the so-called  Neighbourhood Pattern Similarity approach. Our analysis uses results from random matrix theory to show that, when assuming that the graph is generated as a particular Stochastic Block Model with Bernoulli probability distributions for the different blocks, then the recovery is asymptotically correct when the graph has a sufficiently large dimension. Under these assumptions there is a sufficient gap between the dominant and dominated eigenvalues of the similarity matrix, which guarantees the asymptotic correct identification of the number of different roles. We also comment on the connections with the literature on Stochastic Block Models, including the case of probabilities of order $\log(n)/n$ where $n$ is the graph size. We provide numerical experiments to assess the effectiveness of the method when applied to practical networks of finite size.
\end{abstract}



\section{Introduction}

The analysis of large graphs frequently assumes that there is an underlying structure in the graph that allows us to represent it in a simpler manner. A typical example of this is the detection of communities, which are groups of nodes that have most of their connections with other nodes of the same group, and few connections with nodes of other groups. Various measures and algorithms have been developed to identify  community structures \cite{Girvan02}
and many applications have also been found for these model structures \cite{Fortunato:2010,GSSPLNA:2010, Newman:2010,POM:2009}. Yet, many graph structures cannot be modelled using communities: for example, arrowhead and tree graph structures, which appear in overlapping communities, human protein-protein interaction networks, and food and web networks \cite{GSSPLNA:2010,KKKR:2002,PSR:2010}. These more general types of network structures can be modelled as role structures, and the process of finding them is called the role extraction problem, or block modelling \cite{BDVB:2013,DBF:2005,Lei15,Reichardt:2009,RW:2007}.
The role extraction problem is a generalization of the community detection problem and it determines a representation of a network by a smaller structured graph, where nodes are grouped together based upon their interactions with nodes in either the same group or in different groups called roles. If no a priori information is available, one needs to verify all possible group and role assignments in order to determine the best role structure for the data, which leads to an NP-hard problem \cite{BRGGHNW:2006,DBF:2005} for both the community detection problem and the more general role extraction problem.

There are many algorithms proposed for community detection, both for directed and undirected graphs \cite{Duch05,Girvan02,Jing21,Li21,Satuluri11,FL:2009}, but they 
often do not state any conclusive results
about the exact recovery of communities, because they make no statistical assumption about the underlying model of the graph. On the other hand, if one assumes that the adjacency matrix of the graph is a sample of a random matrix that follows certain rules, then the problem of recovering the correct underlying block structure may become tractable. 
The Stochastic Block Model (SBM) is precisely such a model: the interactions between all nodes of a particular group with all nodes of another group follow exactly the same distribution \cite{Holland:1983}. 
There is a considerable literature on SBM \cite{Gao18,Lei15,Zhang16}, including variants that address diagonal scaling of the SBM \cite{Qin13}.

To deal with this problem, researchers
have proposed a variety of procedures, which vary greatly in their degrees
of statistical accuracy and computational complexity. See for example
modularity maximization \cite{NG:2004}, 
likelihood methods \cite{Blondel08,BC:2009,ACW:2012,ACL:2013,CDP:2012},
Infomod methods \cite{BR:2008, BR:2011},
Monte Carlo methods \cite{LMRYZ:2011,P:2014},
method of moments \cite{AGHK:13}, 
belief propagation \cite{Decelle:2011}, 
convex optimization \cite{CSX:2012} and its variants \cite{C:2010,CCT:2012},
methods based on mixture models \cite{LN:2007,MR:2008},
the clique percolation method \cite{PDFV:2005},
spectral embeddings \cite{ALPST:2014}
and hierarchical clustering through minimum description length \cite{P:2013,P:2014X,BR:2007} 
or Bayesian model selection \cite{CL:2015,MRV:2010}.

A class of algorithms that has been largely employed in the past years for such purpose are the so-called spectral methods  
\cite{Karrer11,V:2007,BKSX:2011,FPST:2012, FPSTV:2013}. Broadly speaking, a spectral method first performs an eigendecomposition of a symmetric matrix encoding the properties of the graph. Then the community membership is inferred by applying a clustering algorithm, typically $K$-means, to the rows of the matrix formed by the first few leading eigenvectors. Spectral clustering is easier to implement and computationally less demanding than many other methods, which amount to computationally intractable combinatorial searches. From a theoretical standpoint, spectral clustering
has been shown to enjoy good theoretical properties in stochastic
block models \cite{Rohe11,J:2015,BS:2015}. In the
computer science literature, spectral clustering is also a standard procedure
for graph partitioning and for solving the planted partition model, a special
case of the SBM \cite{JNW:2001}.

As their first step requires the eigendecomposition of a symmetric matrix, spectral methods are commonly applied to undirected graphs.
Moreover, when they do consider 
directed graphs, their analysis does not include the recovery of the underlying block structure \cite{Nowicki01}.

In this paper, we will show that a particular method, using the so-called Neighbourhood Pattern Similarity (NPS) matrices \cite{BVD:2014,Marchand21}, allows us to give a positive answer to the following question: Can we recover asymptotically the block structure of a general directed graph with a stochastic block model structure? A NPS matrix is a real symmetric positive semi-definite matrix,  also for directed graphs, and therefore has real eigenvalues and eigenvectors. 
We then show that for sufficiently large graphs, the gap between dominant and dominated eigenvalues allows a convergent recovery of which nodes are associated with the different roles in the model. 
The nearest results available in the literature are the successful extraction of the correct roles in SBM for the community detection problem of undirected graphs \cite{Lasse3,Lei15}, and the use of spectral clustering for the directed role extraction problem \cite{Qing21}, in which a different type of stochastic block model is used. 
The present paper extends the asymptotic analysis of the general role modelling problem to specific symmetrizations of the standard SBM model for directed graphs, the NPS matrices, for which a correctness result is still missing in the literature. This is in particular one of the few existing results of correctness for the spectral clustering algorithm applied to directed graph with a SBM structure. Furthermore, our results can be seen as covering a whole class of methods, in the sense that our asymptotic analysis applies to all the admissible values of the scaling factor $\beta$ and all NPS matrices $S_k$ (including both any finite value of $k$ and the limit $S=\lim_{k \rightarrow \infty} S_k$, which is the NPS matrix); see Subsection \ref{sec:roleextractionbyS} for the definitions of $\beta$ and $S_k$. While in this paper we focus on the theoretical analysis of the method, determining optimal values of $k$ and $\beta$ for a practical implementation of the algorithm to analyse actual graphs is an interesting possible subject of future research.

In Section \ref{sec:preliminaries} we go over several preliminaries related to graphs, random matrices, stochastic block models and role modelling. 
Section \ref{sec:Spectral_bounds} then yields the spectral bounds for the NPS matrix associated with the graph (as well as for the matrices $S_k$ whose limit is the NPS matrix) and in Section \ref{sec:Clustering_error} we describe the asymptotic behaviour of the clustering error.  
In Section \ref{sec:numerics} we give a few numerical experiments illustrating our theoretical analysis, and we conclude with a few final remarks in Section \ref{sec:conclusion}. 
Several technical proofs are moved to the Appendices for the sake of  readability.

\section{Preliminaries} \label{sec:preliminaries}

\subsection{Graph theory and role extraction}

An unweighted \emph{directed graph}, or digraph, $G=(V,E)$ is an ordered pair of sets where the elements of $V=[n]$ are called \emph{vertices}, or nodes, while the elements of $E \subseteq V \times V$ are called \emph{edges}, or links.
A walk of length $\ell$ on the digraph $G$ from $i_1$ to $i_{\ell+1}$ is a sequence of vertices of the form $i_1,i_2,\dots,i_{\ell+1}$ such that for all $j=1,\dots,\ell$, $(i_j,i_{j+1}) \in E$. $G$ is said to be strongly connected if for all $i,j \in V$ there is a path on $G$ from $i$ to $j$.

The adjacency matrix of an unweighted digraph is defined as 
\[ 
A \in 
\R^{n \times n},
\qquad A_{ij} = 
\begin{cases}
1 \ &\text{if} \ (i,j) \in E;\\
0 \ &\text{otherwise}.
\end{cases}
\]
In particular, $A$ is a non-negative matrix and 	
$A \in \{0,1\}^{n \times n}$.
It is well known that $A$ is irreducible if and only if $G$ is strongly connected; in that case, the Perron-Frobenius spectral theory for irreducible non-negative matrices applies. Manifestly, there is a bijection between adjacency matrices and digraphs. Moreover, two digraphs are isomorphic (i.e. they coincide up to a relabelling of the vertices) if and only if their adjacency matrices are permutation similar.

Given a graph $G$ with adjacency matrix $A$, the problem of role extraction consists in finding a positive integer $r \leq n$ and an \emph{assignment} function $\xi:[n] \rightarrow [r]$ such that $A$ can be well approximated  by an ideal adjacency matrix $E$ such that $E_{ij}$ only depends on $\xi(i)$ and $\xi(j)$. Equivalently, if $\pi$ is any permutation that reorders the nodes such that nodes of the same group are adjacent to each other, and $P$ is the corresponding permutation matrix, then $P^T A P$ is approximately constant in the $r$ blocks induced by the assignment. One can then associate with it a so-called {\em ideal} adjacency matrix $A_{id}$ as illustrated in Figure 1~: for the blocks of $A$ where 1 dominates, put all elements of the corresponding block in $A_{id}$ equal to 1 and
for the blocks where 0 dominates, put them all equal to 0 in $A_{id}$. In doing so, the nodes in each block of $A_{id}$ are {\em regularly equivalent} \cite{EB:1994}, i.e. they have the same parents and the same children \cite{BVD:2014,Marchand21}.
The above approximation problem for $P^TAP$ can thus be viewed as finding a nearby regularly equivalent 
graph to a given graph. 

\begin{figure}[ht] \label{fig:ideal}
	\begin{center}
		\includegraphics[width=.5\textwidth ]{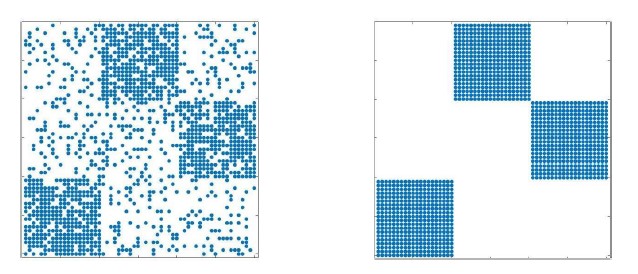} 
		\caption{Associating a regularly equivalent graph $A_{id}$ to the permuted graph $P^TAP$} 
	\end{center}
\end{figure}

The role extraction problem can be also generalized into finding clusters of nodes $\mathcal C_1,\dots,\mathcal C_q$, or equivalently an assignment function, such that for a given node $x$ in cluster $\mathcal C_a$, the number of edges between $x$ and the cluster $C_b$ depends only on $a$ and $b$. In the next section we will find that this problem is better formalized by the  Stochastic Block Model.

\subsection{Stochastic Block Model}\label{sec:sbm}

Let $(\Omega,\mathcal{F},\P)$ be a probability triple and consider the space of random variables 
$\Omega \rightarrow \{0,1\}^{a \times b}$.
A random digraph $G(\omega)$, $\omega \in \Omega$, is a graph whose adjacency matrix $A(\omega)$ is one such random variable. We denote the expectation of a random matrix $A$ by $\E[A]$.
We construct a random unweighted digraph as follows:
\begin{enumerate}
	\item The nodes are partitioned in $q$ clusters of nodes, $\mathcal C_1, \mathcal C_2, \dots, \mathcal C_q$, of size $m_1 n$, $m_2 n$, $\dots,$ $m_q n$ respectively.
	\item There is an edge between a node in cluster $\mathcal C_a$ to a node in cluster $\mathcal C_b$ with probability $p_{a,b} = f(n) \theta_{a,b}$, where $\theta_{a,b}$ depend only on $a$ and $b$ and $\max_{a,b} \theta_{a,b}=1$.
\end{enumerate}

Since $p_{a,b}$ are probabilities, necessarily $f(n) = O(1)$, but from classical information theory we know that
exact recovery for the clusters requires \cite{Abbe16,Mossel15}
\begin{equation}\label{eq:fn}
	nf(n) \to \infty,
\end{equation}
and it is also more restrictive than the sufficient condition for clustering detection \cite{Massoulie14}. 

The adjacency matrix $A_n$ of such a random graph is an $mn \times mn$ random matrix, where $m=\sum_{i=1}^q m_i$.
Suppose that $m_i$ may vary with $n$, but for every $n,i$ we have $0< m_{min}\le m_i\le m_{max}$ where $m_{min}, m_{max}, q, \theta_{a,b}$ do not depend on $n$. As a consequence $m$ may vary, but it is always bounded between absolute constants $qm_{min}\le m\le qm_{max}$. 
Suppose $i \in \mathcal C_a, j \in \mathcal C_b$: then, $A_{ij}$ is distributed as a Bernoulli variable centered on $\{0,1 \}$ with $\P(1)=p_{ab}$. 
In this section, we assume that the nodes of the same cluster are adjacent to each other, in order to simplify the notation. This does not affect the generality of our results. 

Denoting by $ \mathbf{1}_k \in \R^k$ the vector of all ones and by $M_n = \E[A_n]$, then
\[  M_n=f(n) Z_n \Upsilon Z_n^T    \]
where 	
\begin{align*}
	Z_n = &\bigoplus_{i=1}^q \mathbf{1}_{m_in} \in \R^{mn \times q}\\ \Upsilon \in \R^{q \times q}, \quad & \Upsilon_{ab}=\theta_{ab} \ \forall \ 1 \leq a,b \leq q.
\end{align*}
$M_n$ is a deterministic matrix with precisely $s:= \rank(f(n)\Upsilon)\le q$ nonzero singular values:
if $D = \diag(\sqrt {m_i},\dots,\sqrt{m_q})$, then $\wt Z_n := Z_nD^{-1}/\sqrt n$ has orthogonal columns and the nonzero singular values of $M_n$ are those of $nf(n)D\Upsilon D$. We have in particular that
\begin{align}\label{eq:Mn}
	m_{min}  \le
	\frac{\sigma_i(M_n)}{ nf(n) \sigma_i(\Upsilon)}
	\le  m_{max}&\qquad  i=1,\dots,s, \\
	\nonumber \quad \sigma_i(M_n) = 0& \qquad \forall \ i >s.
\end{align}
Analogously, $[M_n \,\, M_n^T]$ has precisely $r:= \rank(f(n)[\Upsilon\,\,\Upsilon^T])\le q$ nonzero singular values with
\begin{align}\label{eq:MMt}
	m_{min}  \le
	\frac{\sigma_i([M_n \,\, M_n^T])}{ nf(n) \sigma_i([\Upsilon\,\,\Upsilon^T])}
	\le  m_{max}&\qquad  i=1,\dots,r, \\
	\nonumber \sigma_i(M_n) = 0 &\qquad \forall \ i >r.
\end{align}
The above scenario is what arises in the theory of Stochastic Block Model, but in most references the matrix $\Upsilon$ is taken symmetric. In the following sections we will analyse the model described above, together with a spectral method designed to extract the clusters, which will be called roles, through the use of a similarity matrix $S$. For this reason, we report here a result we will need in our arguments about the matrix $Y_n:= A_n - M_n$.

\begin{theorem}\label{thm:normtao}\cite[Remark 5.19]{Bai2}\cite[Corollary 2.3.5]{taoblog}
	Let $E_N$ be $N \times N$ random matrices with independent, mean zero, and uniformly bounded entries. Suppose that $\sigma^2$ bounds the second moments of all entries, independently on $N$. In this case, 
	\[
	\limsup_{N\to\infty} \left\| \frac 1 {\sqrt N} E_N \right\|\le 2\sigma 
	\]
	almost surely. 
\end{theorem}
Since the entries of $E_n = Y_n/\sqrt{f(n)}$ have variance bounded by $\max_{i,j} \Upsilon_{i,j}=1$, we get that  
\begin{equation}\label{eq:Y}
	\|Y_n\|^2\le \delta^2  :=4mnf(n)
\end{equation}
when $n$ is big enough.
In what follows, we will  bound the norm of $[Y_n\,\, Y_n^T]$ with $\sqrt 2 \delta = 2\sqrt 2 \sqrt {mnf(n)}$, but
the constant $2\sqrt 2$ here is not sharp. In fact, both Theorem 4.1 in \cite{dep} and the experiments we will present suggest that the result holds with the tighter constant $1+\sqrt 2$, following the classical bound on the Marchenko-Pastur distribution. 
Since there is no such result in literature, we formulate it here as a conjecture.

\begin{conjecture}\label{conj}
	Let $Z_N$ be $N \times N$ random matrices with independent, mean zero, and uniformly bounded entries. Suppose that $\sigma^2$ bounds the second moments of all entries, independently on $N$. If we call $X_N: = [Z_N\,\,Z_N^T]$, then  
	\[
	\limsup_{N\to\infty} \left\| \frac 1 {\sqrt {2N}} X_N \right\|\le \left(1 + \sqrt{\frac 12} \right)\sigma 
	\]
	almost surely. Moreover, if every entry has the same second moment $\sigma^2$, the bound is attained. 
\end{conjecture}

From now on, when we say "for any $n$ big enough" or a similar formulation, we always implicitly mean that the result holds almost surely.

\subsection{Role extraction via the similarity matrix $S$}\label{sec:roleextractionbyS}

In \cite{BVD:2014,Marchand21}, it was proposed to solve the problem of role extraction for a digraph with adjacency matrix $A$ by means of a Neighbourhood Pattern Similarity matrix $S$, which is defined as the limit of the sequence of SPD matrices $(S_k)_{k \in \mathbb{N}}$ with
\begin{equation}\label{eq:recurrence}
	S_0 = 0, \quad S_{k+1} = \Gamma_A[I+\beta^2 S_k]=S_1+\beta^2\Gamma_A[S_k] ,
\end{equation} 
where the operator (depending on the matrix parameter $W$) $\Gamma_W$ is defined as
\begin{equation*}
	\Gamma_W[X]=WXW^T + W^T X W.
\end{equation*}
It was shown in \cite{BVD:2014, Marchand21} that the sequence is convergent if and only if $\beta^2 < \rho(A \otimes A^T + A^T \otimes A)^{-1}$, and that the limit $S$ satisfies $S=S_1+\beta^2\Gamma_A(S)$ or, equivalently, 
\[ \mvec{S}=(I-\beta^2 A \otimes A^T - \beta^2 A^T \otimes A)^{-1} \mvec(AA^T+A^TA).\]
It was also shown there that element $(i,j)$ of the matrix $S_k$ is the weighted sum of the walks of length up to $k$ between nodes $i$ and $j$, and that this can exploited to find nodes that should be associated with the same role. Note that $S_1$ is a known symmetrization for direct graphs called ``Bibliometric Symmetrization'' \cite{Satuluri11}.

Throughout this document, the parameter $\beta$ in \eqref{eq:recurrence} is always assumed to satisfy $\beta^2\|\Gamma_A\|<1$, which is sufficient for the sequence $(S_k)_{k \in \mathbb{N}}$ to converge. Here and thereafter, we measure the norm of the operator $\Gamma_W$ induced by the spectral norm of its matrix argument. More concretely,
\begin{equation*}
	\|\Gamma_W\| = \sup_{X \neq 0}  \frac{\| \Gamma_W[X]\|}{\|X\|}.
\end{equation*}

\begin{lemma} \label{bound}
	The norm 
	of the linear matrix mapping $\Gamma_W: X \mapsto\Gamma_W[X]$  satisfies
	$$ \| \Gamma_W \|
	= \| \left[\begin{array}{cc}W & W^T\end{array}\right] \|^2
	\le  2\|W\|^2.
	$$
\end{lemma}
From the previous result, whose proof can be found in the appendices, $\|\Gamma_W\|\le 2\|W\|^2$, so that it is easy to compute a good enough $\beta$ with very low computational effort. In fact, we can always choose, for example, $\beta^2 = 1/(4 \|A\|^2)$ and obtain that necessarily $\beta^2\|\Gamma_A\| \le 1/2$.\\

Consider the Stochastic Block Model described in Section	\ref{sec:sbm}.
From now on, we drop for the sake of notational simplicity the suffixes emphasizing the dependence on the size $n$, so, for example, we simply write $A,M,Z,Y$ for $A_n,M_n,Z_n,Y_n$.
Given the random adjacency matrix $A$, suppose that $f(n)\Upsilon$ is a minimal role matrix, defined as follows.

\begin{definition} \label{def:min} 
	A square matrix $B$ is a \textbf{minimal role} matrix if no two rows of the compound matrix  $\left[\begin{array}{cc} B & B^T  \end{array}\right]$ are linearly dependent.
\end{definition}

\noindent The matrix $M=\E[A]=f(n)Z\Upsilon Z^T$ is a deterministic block matrix, and the following result  shows that it is possible to recover the original clusters by analysing any of the matrices $T_k$ generated as the $S_k$  but replacing $A$ with $M$.

\begin{theorem}\cite[Theorem 3.4]{Marchand21}\label{thm:lowrankfactor}
	Let $M, \Upsilon$ be as in Section \ref{sec:sbm} with minimal role matrix $f(n)\Upsilon$. If $T_k$ is generated by the recurrence
	\begin{equation*}
		T_0 = 0, \quad T_{k+1} = \Gamma_M[I+\beta^2 T_k]=T_1+\beta^2\Gamma_M[T_k] ,
	\end{equation*} 	
	then it has rank $r = \rank(f(n)[\Upsilon \,\,\Upsilon^T])\le q$ and
	$T_k = Z\wh T_kZ^T$ where $\wh T_k$ is a SPD $q\times q$ matrix.
	Given  $V_k\in\mathbb{R}^{mn\times r}$ the orthogonal matrix in the reduced SVD (or, equivalently, reduced eigendecomposition) of $T_k$, it follows that 
	the set of the vectors of $\mathbb{R}^r$ that are a row of $V_k$ has precisely $q$ distinct elements. Moreover, the $q$ original  clusters of the graph coincide with the partition of $[mn]$ into the $q$ subsets associated with the row indices that correspond to each distinct vector that is a row of $V_k$.
\end{theorem}

\noindent As a consequence, it is enough to perform an eigendecomposition of $T_k$, extract the reduced orthogonal matrix $V_k$ and then identify the repeated $q$ rows to recover the clustering. 
A natural thought is to try and apply the same method to the random symmetric matrix $S_k$ generated by the recurrence \eqref{eq:recurrence}, but some issues arise.
\begin{itemize}
	\item $T_k$ has rank $r\le q$, while $S_k$ is with high probability full rank, so we need a way to determine the truncation parameter $r$ for the SVD.
	\item In the truncated eigendecomposition of $S_k$, the orthogonal matrix $U_k$ has usually distinct rows. In order to retrieve the clusters, we thus need to estimate the number of roles $q$ and perform a $K$-means algorithm on the rows.
\end{itemize}
There is method to do this, commonly referred to as Spectral Clustering of the matrix $S_k$. A detailed description is given in \cite{Marchand21} and a concise description as pseudocode is given below:
\begin{algorithm} \noindent
	\begin{itemize}
		\item Inputs: adjacency matrix $A$, number of roles $q$, scaling factor $\beta$, integer $k$. 
		\item Output: a partitioning of the nodes of the graph into $q$ clusters.
		\item  Procedure:
		\begin{itemize}
			\item Compute the matrix $X_1$ whose columns are the $q$ dominant singular vectors of $\begin{bmatrix}
			A&A^T
			\end{bmatrix}$;
			\item For $h=2,..,k$ compute the matrix $X_h$ whose columns are the $q$ dominant singular vectors of $Y_h=\begin{bmatrix}
			\beta A X_{h-1} & \beta A^T X_{h-1} & X_1
			\end{bmatrix}$;
			\item Apply the $K$-means algorithm to the rows of the matrix $X_k$.
		\end{itemize}
	\end{itemize}
	
\end{algorithm}

The sparse singular value decomposition of $[A,A^T]$ can be computed using the Lanczos bidiagonalization procedure \cite{GoV} and its complexity is $\mathcal{O}(\mu q^2)$ because each matrix vector multiplication requires exactly $2\mu$ flops, where $\mu$ is the number of edges in the graph, i.e., the number of nonzero entries of the $mn \times mn$ matrix $A$. For the same reason, the construction of the matrix $Y_h$ requires exactly $2(\mu+mn)q$ flops.
The singular value decomposition of the economy size singular value decomposition of the dense $mn\times 3q$ matrix $Y_h$, requires $\mathcal{O}(mnq^2)+\mathcal{O}(q^3)$ flops \cite{TrefethenBau}. Altogether, we thus have a complexity of $\mathcal{O}(k q(\mu+mn+mnq+q^2))$ to compute the low rank factor $X_{k}$, which scales well with $\mu$. The subsequent clustering of the rows of $X_{k}$ is then constrained to a $q$-dimensional space, and requires on average 
$\mathcal{O}(mnq^2)$ flops per iteration of the $K$-means algorithm \cite{Kmeans}.

In the next sections, we show that the matrices $S_k$ sport a clear gap between the eigenvalues $\lambda_r(S_k)$ and $\lambda_{r+1}(S_k)$ that lets us identify the rank $r$ with high probability for big $n$. 
Moreover, when the matrix $f(n)[\Upsilon \,\,\Upsilon^T]$ is full-rank, so that $f(n)\Upsilon$ is minimal and $r=q$, we estimate the clustering relative error for the $K$-means algorithm on $S_k$, and show that it is proportional to $(nf(n))^{-1}$.

\section{Spectral Bounds}\label{sec:Spectral_bounds}

We now consider the recurrence relation using the expected value $M$ rather than $A$ since this yields a good approximation for the $S_k$ matrices. We denote these matrices as $T_k$ and their recurrence is thus given by
\begin{equation} \label{eq:Tk} T_0=0, \quad T_{k+1}=\Gamma_M[I_n+\beta^2T_k], \; k\ge 0.
\end{equation}
Note that in \eqref{eq:Tk} the matrix parameter in the operator $\Gamma$ is set to $M=\E[A]$, in contrast with \eqref{eq:recurrence} where it was set to $A$. Again, the parameter $\beta^2$ is chosen such that $\beta^2\|\Gamma_M\| < 1$, 
which is required for the sequence $T_k$ to converge to $T=\lim_{k \rightarrow \infty} T_k$. In order to choose an appropriate $\beta$ we need an estimation of $\|\Gamma_M\|$ depending only on the matrix $A$.

\begin{lemma}\label{lem:missingbit}
	
	Let $\delta^2 = 4mnf(n)$. For $n$ large enough, it holds 
	\[   \| \Gamma_A - \Gamma_M \| \le  
	\delta^3\|[ \Upsilon\,\,\Upsilon^T]\|/\sqrt 2
	+
	2\delta^2 
	\le \|A\|^2. \]
\end{lemma}

\noindent Using the last result and Lemma \ref{bound}, we find that  for $n$ large enough,
$\beta^2 \|\Gamma_M\| \le 3\beta^2 \|A\|^2$ and $\beta^2 \|\Gamma_A\| \le 2\beta^2 \|A\|^2$ so from now on, we always suppose that $\beta^2 \le 1/6\|A\|^2$ and consequently
\begin{equation}\label{eq:beta}
	\gamma:= \max\{\beta^2\|\Gamma_M\|,\beta^2\|\Gamma_A\|\}\le \frac 12.
\end{equation}

It was pointed out in \cite{Marchand21} that the matrices  $S_k$ and $T_k$ are all positive semi-definite, and that both sequences are ordered in the Loewner ordering :
\begin{equation}\label{eq:Loewner}
	0=S_0 \preceq S_1   \preceq \ldots \preceq S, 
	\qquad 0=T_0 \preceq T_1   \preceq \ldots \preceq T.
\end{equation}
Moreover, as shown in Theorem \ref{thm:lowrankfactor}, if $T_k$ were available then we would be able to recover exactly the original clustering that generated the random directed graph. Since we can only work on $S_k$, that is an approximation of $T_k$, it is essential to analyse the proximity between the two matrices more accurately. This will let us study how well the properties of $T_k$ transfer to $S_k$ and how effective is a spectral clustering algorithm applied to $S_k$.

\begin{theorem}\label{deltabound}
	For $k\ge 1$ it holds
	\begin{align*} 
		\|S_{k}-T_{k}\| &\le  \| \Gamma_A - \Gamma_M \|
		\left(\sum_{i=0}^{k-1} \| \beta^2\Gamma_A\|^i\right)\left(\sum_{i=0}^{k-1} \| \beta^2\Gamma_M\|^i\right) \\
		&\le
		4 \| \Gamma_A - \Gamma_M \|, 
	\end{align*}
	where the last inequality holds also for $\|S-T\|$.
\end{theorem}

\subsection{Spectral Gap}

From Theorem \ref{thm:lowrankfactor} we know that $T_k$ has rank $r$, so it stands with reason to expect that its approximation $S_k$ has $r$ dominant singular values (which we refer to as the ``signal") and $mn-r$ small singular values (which we refer to as the ``noise"). 
Here we report estimations for the eigenvalues of $S_k$ and $T_k$ and then derive bounds on the respective gaps.

\begin{lemma}\label{res:signal_eig_Tk_T} 
	It holds 
	$$\lambda_r(T) \ge \lambda_r(T_k) \ge  \left[ 
	\frac
	{\sigma_r([\Upsilon\,\,\Upsilon^T])}{4q}  \frac{m_{min}}{m_{max}}
	\right]^2
	\delta^4
	$$
	for every $k\ge 1$ and $n$ big enough.
\end{lemma}
\begin{proof}Easy corollary of \eqref{eq:Loewner} and \eqref{eq:MMt}.
\end{proof}

\begin{theorem}\label{thm:noise_bound_Sk_S}
	It holds
	\[\begin{array}{c}
	\begin{array}{rcccccl}
	\frac{1}{2} (1-\gamma^{k})\|[\Upsilon\,\,\Upsilon^T]\|^2\delta^4& \ge& \|S_k\| &\ge &\lambda_r(S_k) &\ge& 
	\lambda_r(T_k)/2,  \\
	\frac{1}{2}\|[\Upsilon\,\,\Upsilon^T]\|^2\delta^4& \ge& \|S\| &\ge &\lambda_r(S) &\ge& \lambda_r(T)/2, 
	\end{array}\\[1em]
	\begin{array}{cc}
	4(1-\gamma^{k}) \delta^2&\ge  \lambda_{r+1}(S_k),\\
	4\delta^2&\ge  \lambda_{r+1}(S)
	\end{array}
	\end{array}
	\]
	for every $k\ge 1$ and $n$ big enough.
\end{theorem}

The gaps $\lambda_r(S_{k})-\lambda_{r+1}(S_{k})$ and $\lambda_r(S_{k})/\lambda_{r+1}(S_{k})$ between signal and noise, are expected to be large enough to allow for the correct truncation for the SVD of $S_k$, and a correct assignment of the different nodes in each ``role", as we will show in the next section.
This separation becomes more pronounced when the dimensions of the matrix and its subgroups increase, as we can see by applying Lemma \ref{res:signal_eig_Tk_T} and Theorem \ref{thm:noise_bound_Sk_S}.
\begin{itemize}
	\item For the absolute gap,
	\[
	\lambda_r(S_{k})-\lambda_{r+1}(S_{k}) \ge 
	\frac{\lambda_r(T_k)}2 - 	4(1-\gamma^{k}) \delta^2 
	= \Omega(\delta^4) 
	\]
	that is order of magnitudes greater than the following absolute gaps,
	since
	\[
	\lambda_i(S_{k})-\lambda_{i+1}(S_{k}) \le 
	\lambda_{r+1}(S_{k}) = O(\delta^2),
	\qquad \forall i>r.
	\]
	\item 	For the relative gap,
	\[
	\frac{\lambda_r(S_{k})}{\lambda_{r+1}(S_{k})}
	\ge 
	\frac{\lambda_r(T_k)}{8(1-\gamma^{k}) \delta^2 }
	= \Omega(\delta^2)
	\]
	that is order of magnitudes greater than the previous relative gaps,
	since
	\[
	\frac{\lambda_i(S_{k})}{\lambda_{i+1}(S_{k})}\le 
	\frac{\|S_k\|}{\lambda_r(S_k)} = O(1)
	,
	\qquad \forall i<r.
	\]
\end{itemize}
As a consequence, a comparison of the gaps between signal and noise with the other gaps is a clear indicator of the right rank $r$ with which one should perform the truncated SVD in the algorithm. This holds also for the limit matrix $S$.

We can note that all the estimates get worse as $\sigma_r([\Upsilon\,\, \Upsilon^T])$ gets close to zero.
This has to be expected since it is harder to compute the rank of an almost singular matrix. 
In fact, for example, in the case where all the probabilities $\theta_{a,b}$ are close to each other, it is harder to distinguish between different groups and clusters.

\subsection{Dominant Subspaces}

In this subsection, we study the dominant subspace of a real symmetric matrix $S_k$ (i.e. the invariant subspace associated with the largest $r$ eigenvalues) and argue why, for sufficiently large $n$, it allows role extraction. Classically, distances between subspaces are measured via the concept of principal angles \cite{jordan}: a multidimensional generalization of the acute angle between the unit vectors $u,v$, i.e., $0 \leq \theta(u,v):=\arccos |u^Tv| \leq \pi/2$. More generally, if $\mathcal{U}, \mathcal{V}$ are subspaces whose orthonormal bases are given, respectively, as the columns of the matrices $U,V$, then the $\min\{ \dim \mathcal{U}, \dim \mathcal{V} \}$ largest singular values of $U^TV$ are the sines of the principal angles between $\mathcal{U}^\perp$ (orthogonal complement of $\mathcal{U}$) and $\mathcal{V}$.  Just as in the one-dimensional case, the principal angles between two subspaces are all zero if and only if the subspaces coincide, and more generally the smaller the principal angles the closer the subspaces.

To set up notation, fix $k\in \f N$, and let 
$E$ and $F$ be the dominant subspaces of dimension $r$ for $S_k$ and $T_k$ respectively.  By classical results in geometry and linear algebra \cite{daviskahan,ipsen}, the $r$ largest singular values of the matrix
\begin{equation}\label{eq:proj_to_sin}
	\sin \Theta := \Pi_E - \Pi_F  
\end{equation}
are the sines of the principal angles between the dominant subspaces of $T_k$ and that of $S_k$, where $\Pi_E$ and $\Pi_F$ are the orthogonal projection matrices on the relative subspaces. Hence, the spectral norm of $\sin \Theta$ measures how well the dominant subspace of the similarity matrix $S_k$ approximates the one of the ideal graph.

We rely on Davis-Kahan's sine theta theorem \cite{daviskahan}, in the form given by \cite[Theorem 5.3]{ipsen}. Call $\wh S_k$ the best $r$-rank approximation of $S_k$. 
Since the $r$-th eigenvalue of $T_k$ is larger than the $(r+1)$-th eigenvalue of $\wh S_k$ (which is $0$), the assumptions of \cite[Theorem 5.3]{ipsen} apply and thus
\begin{align*}
	\| \sin \Theta \| &\leq \frac{\|T_k-\wh S_k\|}{ \lambda_r(T_k)} \le
	\frac{\|T_k- S_k\|}{ \lambda_r(T_k)} + 
	\frac{\|S_k-\wh S_k\|}{ \lambda_r(T_k)}\\&
	\le
	\frac{2\|T_k- S_k\|}{ \lambda_r(T_k)}.
\end{align*}
\begin{remark}
	We could apply \cite[Theorem 5.3]{ipsen} reverting the roles of $S_k$ and $T_k$, obtaining
	\[
	\| \sin \Theta \| \le  \frac{\| S_k-T_k\| }{\lambda_r(S_k)}. 
	\]
	In this case, though, we prefer to deal with the deterministic quantity $\lambda_r(T_k)$ instead of the aleatory $\lambda_r(S_k)$, even if the estimation gets worse by  a constant factor 2.
\end{remark}
Using the results of the previous section, in turn this yields for sufficiently large $n$
\begin{equation}\label{eq:sin}
	\| \sin \Theta \| \le
	\frac{4\sqrt 2	\delta^3\|[\Upsilon\,\,\Upsilon^T]\|
		+
		16\delta^2  }{
		\left[ 
		\frac
		{\sigma_r([\Upsilon\,\,\Upsilon^T])}{4q}  \frac{m_{min}}{m_{max}}
		\right]^2
		\delta^4
	} 
	= O(\delta^{-1}).
\end{equation}
Therefore, we can state
\begin{corollary}\label{angle}
	Asymptotically as $n \rightarrow \infty$, the principal angles between the dominant subspaces of $S_k$ and $T_k$ tend to $0$ at least as fast as $\delta^{-1}$.
\end{corollary}

\section{Clustering Error} \label{sec:Clustering_error}

In the previous sections we have estimated how close the matrix $S_k$ is to the deterministic matrix $T_k$ and how this influences their spectral properties and their dominant subspaces. Here we show that the same estimates can be used to bound the clustering error of the proposed method on $S_k$, under the technical hypothesis $r=q$ that is, the matrix $f(n)[\Upsilon\,\,\Upsilon^T]$ is full rank. Note that $\Upsilon$ is still allowed to be singular.\\

Recall that the model is generated by the clusters $\mathcal C_1, \dots,\mathcal C_q$, where $\mathcal C_i$ has cardinality $n_i:= m_in$. Suppose that $\mathcal T_1, \dots,\mathcal T_q$ are the resulting clusters from the algorithm operated on the similarity matrix $S_k$. Define the misclassification error $\wh f$ as
\[
\wh f:= \min_{\pi\in\mathcal S_q}\max_{i=1,\dots,q} \frac{|\mathcal T_{\pi(i)} \,\triangle\, \mathcal C_i |}{|\mathcal C_i|}
\]
where $\triangle$ is the symmetric difference of sets defined as the elements belonging to exactly one of the two sets, or equivalently $A\triangle B := (A\setminus B) \cup (B\setminus A)$. $\mathcal S_q$ is the $q$-th symmetric group, that contains all the permutations on $q$ elements.
$\wh f$ is thus a measure of the maximum rate of misclassified points over all clusters, up to the assignment of the correct clusters $\mathcal C_i$ to the  $\mathcal T_i$ derived by the algorithm. In the appendix, we give a proof for the following bound on $\wh f$.

\begin{theorem}\label{thm:misclassification_error}
	There exists an absolute constant $C$ such that asymptotically in $n$ 
	\begin{align*}
		\wh f 
		&\le Cq
		\frac {m_{max}}{m_{min}}  
		\|\sin(\Theta)\|^2
		\le C\frac{q^5}{\delta^2}
		\frac {m_{max}^5}{m_{min}^5}  
		\frac{	\|[\Upsilon\,\,\Upsilon^T]\|^2
		}{
			\sigma_q([\Upsilon\,\,\Upsilon^T])^4
		}\\&= 
		O\left(\frac 1{nf(n)} \right)
		.
	\end{align*}
\end{theorem}

Note that the error goes asymptotically to zero as long as $nf(n)\to \infty$, which is exactly  condition \eqref{eq:fn}.

\begin{remark}
	The proof of the Theorem follows the same steps as  \cite{Qing21} and  \cite{Ads2}. In particular, in the former we find a similar algorithm applied directly on the adjacency matrix $A$ instead of $S_k$, but the analysis is limited to the case where $f(n)\Upsilon$ has full rank, while we work under the more general condition that $f(n)[\Upsilon\,\, \Upsilon^T]$ is full rank. 
	
	Observe moreover that all the results of Section \ref{sec:Spectral_bounds}  hold without any assumption on $\Upsilon$, so we still have all the spectral bounds and the convergence of the dominant subspace of $S_k$ to the one of $T_k$  also in the general case.
	
	Yet, for the algorithm to make sense, we need $f(n)\Upsilon$ to be a minimal role matrix as defined in Definition \ref{def:min}.
	Moreover, since $r<q$, it is necessary to apply the $K$-means algorithm on $S_k$ for $K=r,r+1,r+2,\dots$ and look at the error in order to find the optimal number of clusters.  
	
\end{remark}

\section{Numerical examples} \label{sec:numerics}

In this section we illustrate the theoretical results of the paper using an example generated according to
the rules of a Stochastic Block Model  
where only two different probabilities are used, namely $p$ and $1-p$. 
We chose $q=3$, $m_1=m_2=m_3=10$ and hence $m=30$, and 
\begin{equation}\Upsilon = p \left[\begin{array}{ccc} 0 & 1 & 0 \\ 0 & 0 & 1 \\ 1 & 0 & 0 \end{array}\right] + (1-p) \left[\begin{array}{ccc} 1 & 0 & 1 \\ 1 & 1 & 0 \\ 0 & 1 & 1 \end{array}\right].
\end{equation} 
We then ran simulations for
matrices $A$ with $n=10:50$.

In Figure 2 we took $\beta=0$ which means that the sequences $S_k$ and $T_k$ are constant after one step, and hence that $S=S_1$ and $T=T_1$. The $q$ dominant eigenvalues of $S_1$ are the circles in each plot (because of the structure of $\Upsilon$ there are two repeated ones). The full lines are their estimates obtained from the rank $q$ matrix $T_1$, and clearly they are very accurate estimates as indicated in Theorem \ref{deltabound}. The squares correspond to the ``noise" eigenvalue $\lambda_{r+1}(S_1)$ and the dashed line is its estimate 
$(3+\sqrt{8})p(1-p)mn$ according to Conjecture \ref{conj} and $\lambda_{r+1}(S_1) = \sigma_{r+1}([A\,\,A^T])^2 \le \|[Y\,\,Y^T]\|^2$. It is clear from these plots that this is also a very good estimate and that the ratio 
$\lambda_{r}(S_1)/\lambda_{r+1}(S_1)$ grows like
${\cal O}(n)$. Moreover,  the plots show that the gap  $|\lambda_{r}(S_1)-\lambda_{r+1}(S_1)|$ shrinks with $p$ getting closer to $0.5$, which is expected since for $p=0.5$ the rank of $\Upsilon$ drops to 1. This means that for $p$ getting closer to $0.5$, one has to require
larger dimensions of the graph in order to recover an accurate enough grouping.

\begin{figure}[ht] \label{F2}
	\begin{center}
		\includegraphics[width=8cm, height=6cm]{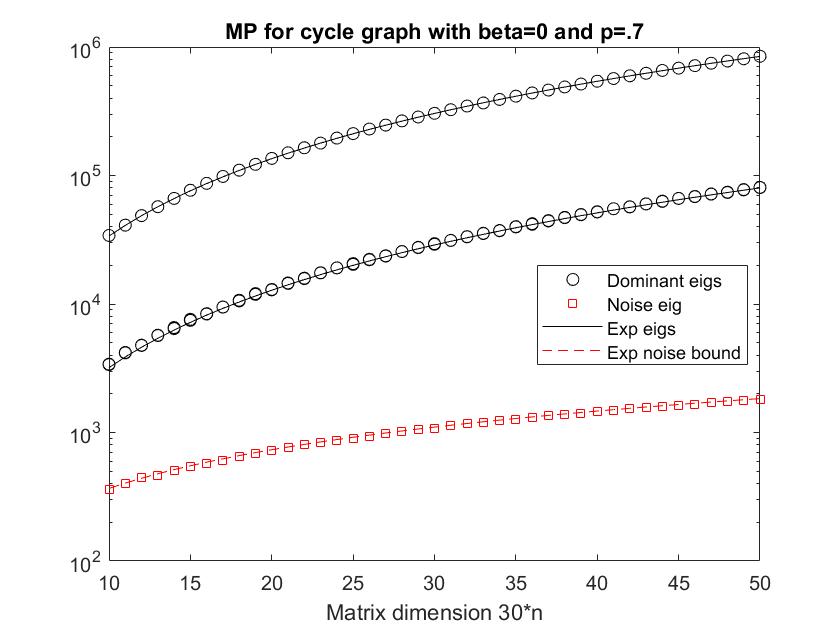}  \includegraphics[width=8cm, height=6cm]{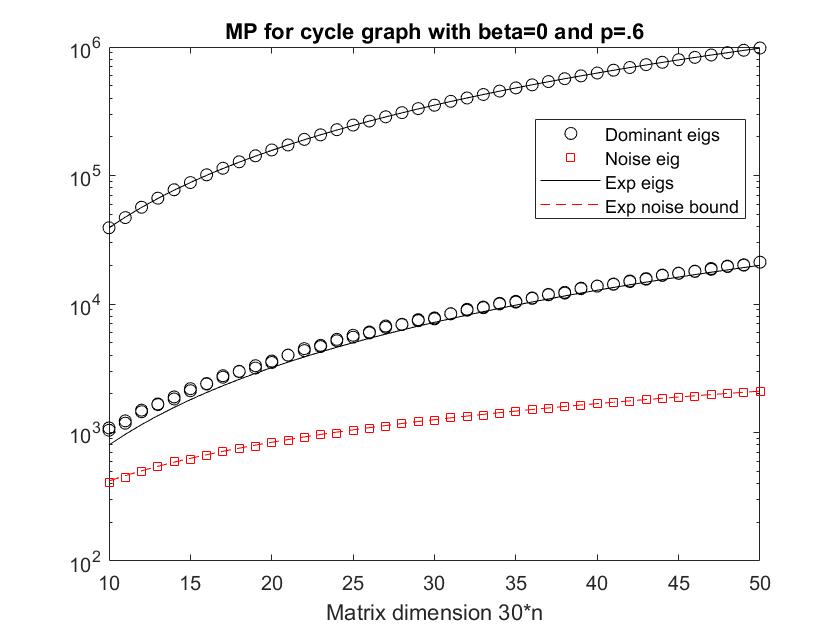}   
		\caption{Eigenvalues of $S_1$ and $T_1$ of a cycle graph for increasing $n$ and varying probabilities.}
	\end{center}
\end{figure}

In Figure 3, we performed the same experiment, but now with $\beta$ chosen such that 
$\|\beta^2\Gamma_A(\cdot)\|\approx \frac12$, which guarantees convergence of the method.
In order to reduce the complexity of the method, we computed $S_{10}$ and $T_{10}$ rather than the limits $S$ and $T$, since in 10 steps we should have reasonably good estimates of these limits. We can see from the plots that 
one has to wait for larger values of $n$ to reach a sufficiently large gap  $|\lambda_{r}(S_{10})-\lambda_{r+1}(S_{10})|$ than for $|\lambda_{r}(S_1)-\lambda_{r+1}(S_1)|$ in Figure 2.

\begin{figure}[ht] \label{beta}
	\begin{center}
		\includegraphics[width=8cm,height=6cm]{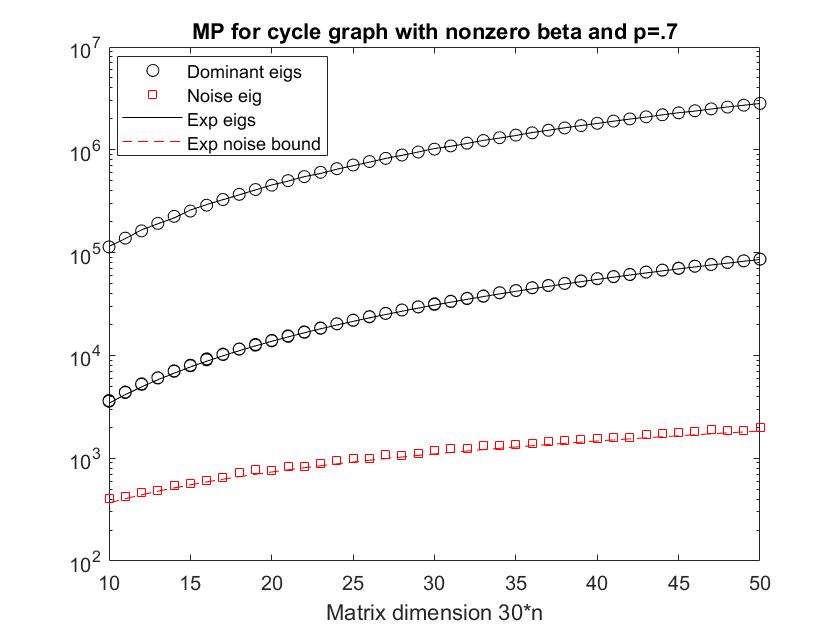}  
		\includegraphics[width=8cm, height=6cm]{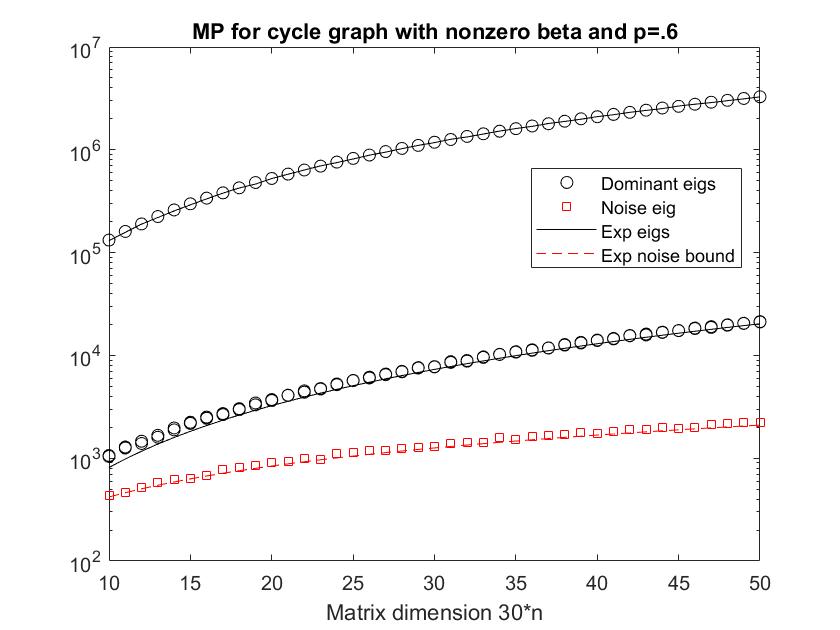}
		\caption{Eigenvalues of $S_{10}$ and $T_{10}$ of a cycle graph for increasing $n$ and varying probabilities}   
	\end{center}
\end{figure}

In Figure 4, we computed the misclassification error $\hat f$ of the clustering associated to the matrix $\Upsilon$ for $p= 0.6$. Using the same parameters $m_i, n$ as before we show an averaged $\hat f$ over 60000 instances for the clusters extracted from $S_1$ and $S_{10}$, where we took $\beta = (2\|[A\ A^T]\|^2)^{-1}$. For comparison, we also plot the function $3/(10n + 24)$ and note that it fits well both plots, thus confirming the bound $O(1/n)$ predicted by Theorem \ref{thm:misclassification_error}.   


\begin{figure}[ht] \label{Miscl_error}
	\begin{center}
		\includegraphics[width=\textwidth]{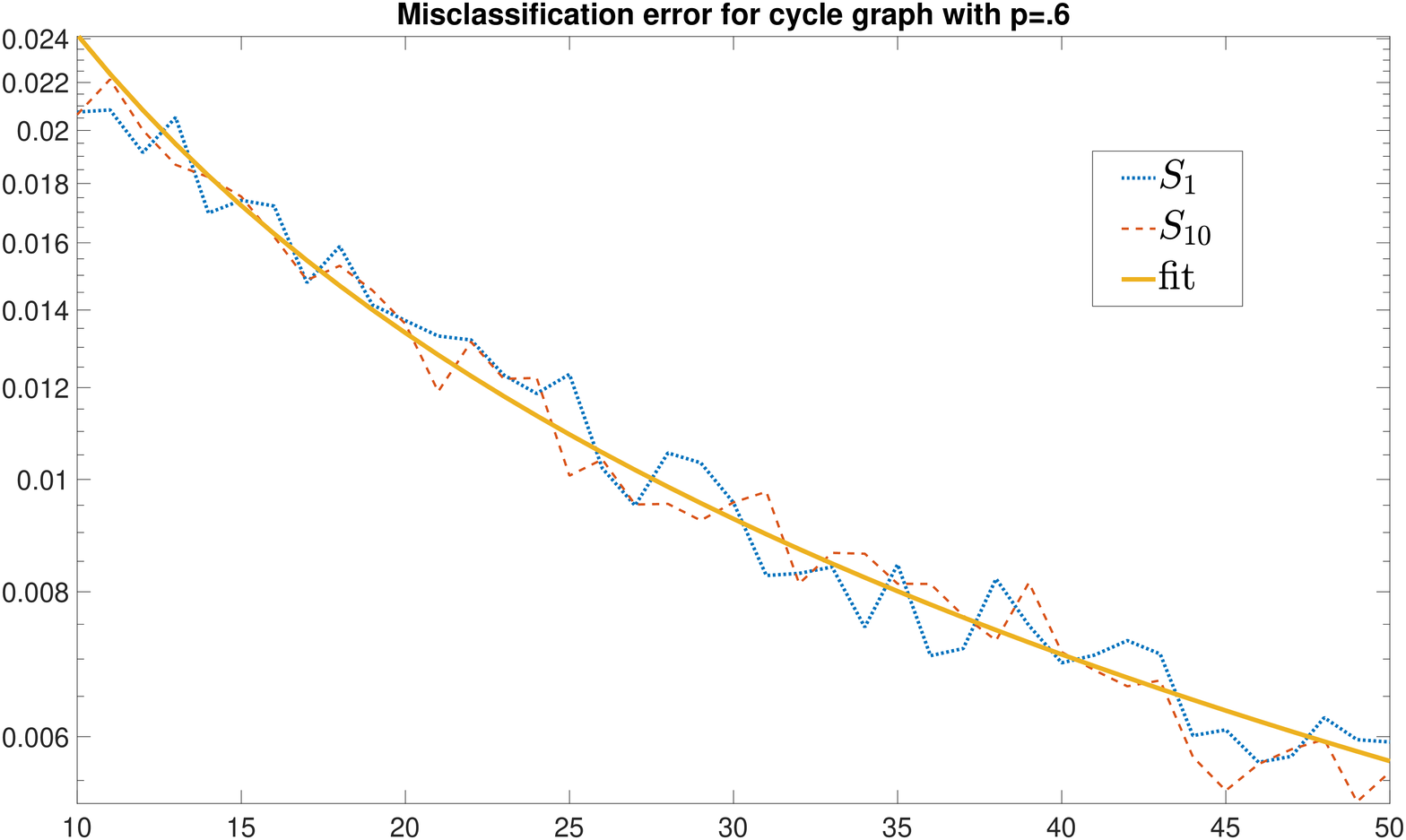}
		\caption{Average misclassification error for the cycle graph relative to $S_1$ and $S_{10}$, and a fitted $O(1/n)$ function for comparison.}   
	\end{center}
\end{figure}

\section{Concluding remarks} \label{sec:conclusion}
In this paper, we showed that the Neighbourhood Pattern Similarity matrices $S_k$ of a directed graph with adjacency matrix $A$ have spectra that are well separated into two groups of eigenvalues, provided that the graph is sufficiently large and that it is generated according to the Stochastic Block Model with blocks where all elements in each block follow a Bernoulli distribution with the same probability $O(f(n))$ where $nf(n)\to\infty$.

The large eigenvalues are then associated with the nonzero eigenvalues of the expected  value $\E[A]$, which is a low rank matrix, and the small eigenvalues are associated with the mean and variance of the random distribution used in the Stochastic Block Model. Moreover, the gap between the ``large" eigenvalues and the ``small" ones grows with $n$. 
It then follows that the recovery of the nodes grouping of the SBM, can be based on the dominant eigenspace of the matrices $S_k$.

This analysis was primarily based on the recurrences defining the matrices $S$ and $T$ and on the fact that the underlying adjacency matrix is generated according to a Stochastic Block Model. It is likely that our results can be extended for other types of distributions and that weighted graphs can also be dealt with, but our analysis here was limited to unweighted adjacency matrices for directed graphs.

We point out that the same analysis could in principle be conducted in the sparse limit case $f(n) = O(1/n)$, but since most of the results are formulated asymptotically in $nf(n)$ one has to  explicitly compute all the implicit multiplicative constants. A technical work of refinement is also needed on each proposition to obtain the best constants and thus meaningful results. For these reasons, we postpone the  limit sparse case analysis to future work.     

\section*{acknowledgments}
	The authors thank the anonymous reviewers for useful comments.	Giovanni Barbarino and Vanni Noferini are supported by an Academy of Finland grant (Suomen Akatemian p\"{a}\"{a}t\"{o}s 331240). Giovanni Barbarino thanks the Alfred Kordelinin säätiö for the financial support under
	the Grant no. 210122. Paul Van Dooren is supported by an Aalto Science Institute Visitor Programme.

\appendix
\section{Operator $\Gamma$}
\subsection{Proof of Lemma \ref{bound}}

Note that
for all $X$
\begin{align*}
	\|\Gamma_W[X]\| &= 
	\left\| \left[\begin{array}{cc}W & W^T\end{array}\right] \left[\begin{array}{cc}X & 0 \\ 0 & X\end{array}\right] \left[\begin{array}{c}W^T \\ W\end{array}\right]\right\| \\&\le
	\| \left[\begin{array}{cc}W & W^T\end{array}\right] \|^2 \|X\|
\end{align*} 
and
\begin{align*}
	\| \left[\begin{array}{cc}W & W^T\end{array}\right] \|^2
	&=
	\left\| \left[\begin{array}{cc}W & W^T\end{array}\right]  \left[\begin{array}{c}W^T \\ W\end{array}\right]\right\|
	\\&= 
	\|WW^T + W^TW\|
	\le 2\|W\|^2
\end{align*}
Thus we have $ \|\Gamma_W\| \le \left[\begin{array}{cc}W & W^T\end{array}\right] \|^2\le 2\|W\|^2$, whose first bound is satisfied for $X=I$ since $$\|\Gamma_W[I]\|/\|I\|
=
\|WW^T+W^TW\|
= \| \left[\begin{array}{cc}W & W^T\end{array}\right] \|^2.$$

\subsection{Proof of Lemma \ref{lem:missingbit}}
For any matrix $X$, if we rewrite $\Gamma_A[X]-\Gamma_M[X]$ as
\begin{align*}
	\begin{bmatrix} Y & Y^T\end{bmatrix}
	\begin{bmatrix} X & 0 \\ 0 & X \end{bmatrix}&
	\begin{bmatrix} M^T \\ M\end{bmatrix} + \begin{bmatrix} M &  M^T \end{bmatrix}
	\begin{bmatrix} X & 0 \\ 0 & X \end{bmatrix}
	\begin{bmatrix} Y^T \\ Y\end{bmatrix}
	\\&+\begin{bmatrix} Y & Y^T \end{bmatrix}
	\begin{bmatrix} X & 0 \\ 0 & X \end{bmatrix}
	\begin{bmatrix} Y^T \\ Y\end{bmatrix}
\end{align*}
we readily see that
\begin{align*}
	\| \Gamma_A-\Gamma_M\|&= \sup_{ X \neq 0}  \frac{  \|  \Gamma_A[X]-\Gamma_M[X] \|}{ \| X \| } \\&\leq \| \begin{bmatrix}
		Y & Y^T
	\end{bmatrix} \|^2 + 2 \| \begin{bmatrix}
		M & M^T 
	\end{bmatrix}\| \| \begin{bmatrix}
		Y & Y^T
	\end{bmatrix} \| . 	
\end{align*}
Using  \eqref{eq:MMt} and \eqref{eq:Y}, and recalling that $\delta^2 =4mnf(n)$,
\[
\begin{array}{rcl}
\|\begin{bmatrix} Y & Y^T\end{bmatrix}\|^2 &\le& 2\|Y\|^2 \le 2\delta^2,\\
\|\begin{bmatrix} M & M^T\end{bmatrix}\|^2	&\le &
\left( m_{max}nf(n)\|[\Upsilon\,\,\Upsilon^T]\|  \right)^2
\\
&\le&
\delta^4\|[\Upsilon\,\,\Upsilon^T]\|^2/16.
\end{array}
\]


Then the desired bound follows easily. As for the last bound, note that $\delta^2 \sim nf(n)$ and in virtue of \eqref{eq:fn}, \eqref{eq:Mn} and \eqref{eq:Y} we have
\[
\|Y\| \le \delta \ll \frac{ \|\Upsilon\|}{4q}\frac{m_{min}}{m_{max}} \delta^2 \le  m_{min}nf(n) \|\Upsilon\| \le \|M\|.
\]
If we call $C$ the constant $\|\Upsilon\|m_{min}/4qm_{max}$, then
\begin{align*}
	\|A\|^2 &\ge \left( \|M\| -\|Y\| \right)^2
	\ge C^2\delta^4 - 2C\delta^3 +\delta^2 \\&\gg 
	\delta^3\|[\Upsilon\,\,\Upsilon^T]\|/\sqrt 2
	+2\delta^2 .
\end{align*}

\section{Spectral Bounds}
\subsection{Proof of Theorem \ref{deltabound}}

Denoting the increments $\Delta^S_{i+1}:=S_{i+1}-S_i$ and $\Delta^T _{i+1}:=T_{i+1}-T_i$  we obtain 
$$  S_{k+1} - T_{k+1} = \sum_{i=1}^{k+1}(\Delta^S_{i}-\Delta^T _{i}), \quad S_0=T_0=0.
$$
Observing that
\begin{align*}
	\Delta^S_{k+1}=\beta^2\Gamma_A[\Delta^S_{k}],& \quad \Delta^S_1=\Gamma_A[I_n], \\
	\Delta^T _{k+1}=\beta^2\Gamma_M[\Delta^T _{k}],& \quad \Delta^T _1=\Gamma_M[I_n],  
\end{align*}
and 
$$   \Delta^S_{k+1}- \Delta^T _{k+1}  = \beta^2 \Gamma_A[\Delta^S_{k}- \Delta^T _{k}] + \beta^2 \Gamma_A[\Delta^T _k] - \beta^2 \Gamma_M[\Delta^T_k]
$$
we can estimate
$$\begin{bmatrix} \|\Delta^S_{k+1}-\Delta^T _{k+1} \| \\ \|\Delta^T _{k+1} \| \end{bmatrix} \le
N
\begin{bmatrix} \|\Delta^S_{k}-\Delta^T _{k} \| \\ \|\Delta^T _{k} \| \end{bmatrix} $$ where $$ N:= \beta^2 \begin{bmatrix} \|\Gamma_A\| & \|\Gamma_A - \Gamma_M\| \\ 0 & \|\Gamma_M \| \end{bmatrix} 
$$
with initial conditions $\|\Delta^S_1-\Delta^T _1 \| =\| \Gamma_A[I_n]-\Gamma_M[I_n]\| \le \| \Gamma_A - \Gamma_M\| $ and $\|\Delta^T _1 \| \le \|\Gamma_M\|$. Hence,  by induction on $k\ge 1$, it is not difficult to obtain the upper bound
\begin{align*} &\begin{bmatrix} \|\Delta^S_{k+1}-\Delta^T _{k+1} \| \\ \|\Delta^T _{k+1} \| \end{bmatrix} \le N^k \begin{bmatrix}  \|\Gamma_A - \Gamma_M\| \\  \|\Gamma_M \| \end{bmatrix}\\
	& \le \beta^{2k}
	\begin{bmatrix}  \|\Gamma_A - \Gamma_M\| \sum_{i=0}^k  \|\Gamma_A \|^i \|\Gamma_M \|^{k-i} \\  \|\Gamma_M \|^{k+1} \end{bmatrix} ,
\end{align*}
which finally yields the bound
$$ \|S_{k}-T_{k}\| \le  \|\Gamma_A - \Gamma_M\|
\left(\sum_{i=0}^{k-1} \| \beta^2\Gamma_A\|^i\right)\left(\sum_{i=0}^{k-1}\| \beta^2\Gamma_{M}\|^i\right).
$$
In virtue of \eqref{eq:beta}, we can let $k$ go to $\infty$ and find that 
\[
\sum_{i=0}^\infty \| \beta^2\Gamma_A\|^i
=
\frac{1}{1-\beta^2\|\Gamma_A\|}\le 2
\]
and the same holds for $M$, thus the desired bound follows.

%

\subsection{Proof of Theorem \ref{thm:noise_bound_Sk_S}}
Note first that  by \eqref{eq:fn}, \eqref{eq:MMt} and \eqref{eq:Y}, for $n$ big enough,
\begin{align*}
	\|[A \,\, A^T]\|&\le \|[M\,\,M^T]\| +\|[Y\,\,Y^T]\| \\&\le \frac {1}4\|[\Upsilon\,\,\Upsilon^T]\|\delta^2 +\sqrt 2\delta \le \frac {1}2\|[\Upsilon\,\,\Upsilon^T]\|\delta^2.
\end{align*}
This can be used to bound $\|S_k\|$ since by the recurrence \eqref{eq:recurrence}, Lemma \ref{bound}, condition \eqref{eq:beta} and induction we find
\begin{align*}
	\|S_{k}\|&\le \|\Gamma_A\|(1+\beta^2\|S_{k-1}\|)
	\le \|\Gamma_A\|\frac{1-\gamma^{k}}{1-\gamma}\\&\le
	\frac{1}{2} (1-\gamma^{k})\|[\Upsilon\,\,\Upsilon^T]\|^2\delta^4,
\end{align*}
and if we let $k\to \infty$ then $\|S\|\le \|[\Upsilon\,\,\Upsilon^T]\|^2\delta^4/2$.
Moreover, 
\[
S_k= \Gamma_A[I+\beta^2S_{k-1}]\preceq 
(1+\beta^2\|S_{k-1}\|) \Gamma_A[I]
\preceq \frac{1-\gamma^{k}}{1-\gamma} S_1
\]
and by Weyl's theorem 
\begin{align*}
	\lambda_{r+1}(S_k) &\le 
	\frac{1-\gamma^{k}}{1-\gamma} \lambda_{r+1}(S_1)
	=
	\frac{1-\gamma^{k}}{1-\gamma} \sigma_{r+1}([A\,\, A^T])^2
	\\&\le 
	\frac{1-\gamma^{k}}{1-\gamma} \|[Y\,\, Y^T]\|^2
	\le 4(1-\gamma^{k}) \delta^2
\end{align*}
where again if we let $k\to \infty$ then $\lambda_{r+1}(S) \le 4\delta^2$. 
Lastly, using \eqref{eq:Loewner} and Weyl's theorem,  
\[\lambda_r(S)\ge
\lambda_r(S_{k}) 
\ge \lambda_r(T_{k}) -  \|S_{k}-T_{k}\|
\]
By Lemma \ref{lem:missingbit}, Theorem \ref{deltabound}, and Lemma \ref{res:signal_eig_Tk_T}, we know that 
$ \lambda_r(T_{k}) =\Omega(\delta^4) \gg O(\delta^3) =  \|S_{k}-T_{k}\|$ so for $n$ big enough, 
$ \lambda_r(S_k) \ge \lambda_r(T_k)/2$ and the same holds for $k\to \infty$.

\section{Clustering Error}

\begin{lemma}\cite[Lemma 8, Appendix C]{Zhixin19}\label{Zhixin19}
	Let $E,F$ be $a\times b$ matrices with orthonormal columns, and let $\Pi_E, \Pi_F$ be the orthogonal projections on their respective ranges. Then there exists an orthogonal $b\times b$ matrix $Q$ such that 
	\[
	\|E-FQ\|_F \le \|\Pi_E-\Pi_F\|_F.
	\]
\end{lemma}

\begin{theorem}\cite[Theorem 1]{br}\label{br}
	Call $\f M$  the set of $nm\times q$ matrices that have only $q$ distinct rows.
	Let $E,F$ be $mn\times q$ matrices, where $F\in \f M$, whose rows $\mu_1,\dots,\mu_q$ identify the  clusters $\mathcal C_i$. 
	Call $n_i:= |\mathcal C_i|$ and 
	$$\Delta_i:= \frac{1}{\sqrt {n_i}}\min\{ \sqrt k\|E-F\|,\|E-F\|_F \}.$$
	Suppose there exists $\rho \ge 100$ such that $\|\mu_i-\mu_j\|\ge \rho (\Delta_i+\Delta_j)$ for any $i\ne j$. 
	Let $G\in \f M$ be a 10-approximation of the $K$-means algorithm on $E$, that is 
	$$\| E-G\|_F^2 \le 10 \min_{N\in \f M} \| E-N\|_F^2 $$
	and call $\nu_1,\dots,\nu_k$ the rows of $G$. 
	Partition the indices $1,\dots,mn$ into $q$ clusters $\mathcal T_i$ according to $G$ and $ E$ as in $\mathcal T_r:=\{ i : 
	\| E_{i,:}-\nu_r\|\le \| E_{i,:}- \nu_s\| \, \forall s \}.$
	We have that there exists a permutation $\pi$ and an absolute constant $C$ such that $|\mathcal C_r\,\triangle\, \mathcal T_{\pi(r)}|\le Cn_r/\rho^2$ for every $r$.
\end{theorem}

\subsection{Proof of Theorem \ref{thm:misclassification_error}}

In order to analyse the method, we need first to better characterize the eigenvalue decomposition (EVD) of $T_k$. In fact, from Theorem \ref{thm:lowrankfactor}, we know that there exists a full-rank PSD matrix $\wh T_k$ such that $T_k=Z\wh T_kZ^T$. Recall now that $D=\diag(\sqrt {m_1},\dots,\sqrt {m_q})$, and that $\wt Z = ZD^{-1}/\sqrt n$ has orthonormal columns. If $D\wh T_k D=U_k \Sigma_k U_k^T$ is its EVD, then
$$
T_k = 
Z\wh T_kZ^T = 
n\wt Z D \wh T_k D \wt Z^T = 
n\wt Z U_k \Sigma_k U_k^T \wt Z^T 
,
$$
where $\wt W_k:= \wt Z U_k = ZD^{-1}U_k/\sqrt n$ has orthogonal columns, so that $\wt W_kn\Sigma_k \wt W_k^T$ is the $q$-reduced EVD of $T_k$.
Note that 
$\wt W_k\in \f M$ since its rows coincide with the ones of $D^{-1}U_k/\sqrt n$, that is a full rank $q\times q$ matrix. For the same reason, we have that $\wt W_k Q\in \f M$ for every orthogonal $q\times q$ matrix $Q$.  Moreover, the clustering induced by $\wt W_k$ and $\wt W_kQ$ are the same, and coincide with the original clustering $\mathcal C_1,\dots, \mathcal C_q$. It follows that if $W_k$, the orthogonal matrix in the $q$-truncated SVD of $S_k$, is close to $\wt W_kQ$ for even one matrix $Q$, then it has good chance to generate a good clustering. Here we report two results formalizing the concept.

By Lemma \ref{Zhixin19} and \eqref{eq:proj_to_sin} there exists a $k\times k$ orthogonal matrix $Q$ such that 
\begin{align*}
	\|\wt W_kQ - W_k\|_F &\le \| \Pi_{\wt W_k} - \Pi_{W_k}\|_F \\&\le \sqrt{2q} \|\Pi_{\wt W_k} - \Pi_{W_k} \| = \sqrt {2q}\|\sin(\Theta)\|
\end{align*}
where $\sin(\Theta)$ 
are the sines of the principal angles between the
subspaces $W_k$ and $\wt W_k$, thus
\begin{align*}
	\Delta_i&= \frac{1}{\sqrt {m_in}}\min\{ \sqrt q\|W_k-\wt W_kQ\|,\|W_k-\wt W_kQ\|_F \} \\&\le \frac{\sqrt{2q}}{\sqrt {m_in}}  \|\sin(\Theta)\|.
\end{align*}
If we call $\mu_1,\dots,\mu_q$ the distinct rows of $\wt W_kQ= ZD^{-1}U_kQ/\sqrt n$, then they are in the form $u_i/\sqrt{nm_i}$ where $u_i$ are the rows of $U_kQ$,  that is  an orthogonal matrix, so

\begin{align*}
	\|\mu_i-\mu_j\|^2 &= \left\|\frac 1{\sqrt{nm_i}}u_i - \frac 1{\sqrt{nm_j}}u_j\right\|^2\\
	&=
	\frac 1{nm_i} + \frac 1{nm_j} = \rho^2(\Delta_i + \Delta_j)^2
\end{align*}
where 

\begin{align*}
	\rho &=
	\frac{
		\sqrt{
			\frac 1{nm_i} + \frac 1{nm_j} }
	}{
		\Delta_i + \Delta_j
	}\ge 
	\frac{
		\sqrt{
			\frac 1{m_i} + \frac 1{m_j} }
	}{
		\frac{1}{\sqrt {m_i}}  +
		\frac{1}{\sqrt {m_j}}
	}
	\frac{
		1
	}{
		\sqrt{2q} \|\sin(\Theta)\|
	}
	\\&\ge 
	\sqrt{
		\frac {m_{min}}{m_{max}}  }
	\frac{
		1
	}{
		2\sqrt{q} \|\sin(\Theta)\|
	}.
\end{align*}

\noindent By Corollary \ref{angle}, $\|\sin(\Theta)\| = O(\delta^{-1})$, so $\rho>100$ for $n$ big enough. The $K$-means algorithm applied to the matrix $W_k$
outputs the clusters $\mathcal T_1,\dots, \mathcal T_q$ and Theorem \ref{br} assures us that there is an absolute constant $C$ for which 

\[
\wh f= \min_{\pi\in\mathcal S_q}\max_{i=1,\dots,q} \frac{|\mathcal T_{\pi(i)} \,\triangle\, \mathcal C_i |}{|\mathcal C_i|}
\le \frac C{\gamma^2}
\le 
4Cq
\frac {m_{max}}{m_{min}}  
\|\sin(\Theta)\|^2
\]

We can finally conclude that by \eqref{eq:sin} and  incorporating all the absolute constants into $C$,

\[
\wh f
\le 
C\frac{q^5}{\delta^2}
\frac {m_{max}^5}{m_{min}^5}  
\frac{	\|[\Upsilon\,\,\Upsilon^T]\|^2
}{
	\sigma_q([\Upsilon\,\,\Upsilon^T])^4
} 
.
\]

\bibliographystyle{plain}
\bibliography{PREref}

\end{document}